\documentclass[a4paper,11pt]{article}
\usepackage{amsmath}
\usepackage{amsthm}
\usepackage[all]{xy}
\usepackage[english]{babel}
\usepackage[utf8]{inputenc}
\usepackage{graphicx}
\usepackage{xypic,color, xcolor}
\usepackage{amsfonts, amssymb}

\newtheorem{thm}{Theorem}[section]
\newtheorem{defn}[thm]{\textbf{Definition}}
\newtheorem{example}{Example}[section]

\newtheorem{cor}[thm]{Corollary}
\newtheorem{lem}[thm]{Lemma}
\newtheorem{prop}[thm]{Proposition}
\theoremstyle{definition}

\theoremstyle{definition}

\theoremstyle{remark}
\newtheorem{rem}[thm]{Remark}

\renewcommand{\k}{\Bbbk}

\newcommand{\de}{\Delta}

\newcommand{\ot}{\otimes}
\newcommand{\ma}{\mathcal}
\newcommand{\rt}{\rightarrow}
\newcommand{\tg}{\ensuremath{\mathrm{t}}}
\newcommand{\st}{\ensuremath{\mathrm{s}}}
\newcommand{\len}{l}

\title{Handle element on nearly Frobenius algebras}
\author{ Dalia Artenstein, Ana Gonz\'alez and Gustavo Mata.
}

\date{\today}

\begin{document}

	\maketitle
	
	\begin{abstract}
	In this article the concept of handle element of Frobenius algebras, as in \cite{K}, will be extended to nearly Frobenius algebras.  The main properties of this element will be analyzed in this case and many examples will be constructed. Also the Casimir and Schur elements of symmetric algebras will be considered (\cite{B}) and generalized for Frobenius and nearly Frobenius algebras showing which results still hold in this framework and which ones fail.
		
	\end{abstract}

\section{Introduction}

The notion of (commutative) nearly Frobenius algebra was introduced in the thesis of
the second author of this article. It was motivated by the result proved in \cite{CG}, which states
that: the homology of the free loop space $HL^{\ast}(LM)$ has the structure of a Frobenius algebra without
counit. Nearly Frobenius algebras were studied in \cite{GLS}, and more recently from an algebraic point of view in \cite{AGL15}, \cite{AGM1}, \cite{AGM2}, \cite{AGM3} and \cite{Gu}. 

The handle element (also called characteristic element or distinguished element) of a Frobenius algebra was introduced by L. Abrams in his thesis (see \cite{A0}). There, he proved the following two results, that show it is an useful concept:\\

{\bf Theorem 2.3.3 of \cite{A0}.} A Frobenius algebra is semisimple if and only if the handle element is a unit.\\

{\bf Proposition 2.3.4 of \cite{A0}.} In a Frobenius algebra, the ideal generated by the handle element coincides with the socle of the algebra.\\

The notion of Schur element of an absolutely irreducible representation of a symmetric algebra was first introduced by M. Geck in \cite{Ge}. In \cite{B} a slight generalization of that notion was presented for an epimorphism between two symmetric algebras. In that article the following result was proved:\\

{\bf Theorem 4.6 of \cite{B}} If $\lambda: A \rightarrow B$ is an epimorphism between symmetric algebras, the following properties are equivalent.
\begin{enumerate}
		\item The Schur element $s_{\phi}$ is invertible in $B$.
		\item $\phi:A\twoheadrightarrow  B$ is split as a morphism of $A$-$A$-bimodules.
		\item $B$ is a projective right $A$-module.
		\item Any projective $B$-module is a projective right $A$-module.
	\end{enumerate}

In this article we study the previous concepts, and it is organized as follows. 

In section 3 we introduce the notion of handle element for nearly Frobenius algebras. We prove an analogous to Theorem 2.3.3 from \cite{A0}: an algebra $A$ admits a nearly Frobenius structure where the handle element is a unit of $A$ if and only if $A$ is separable (Theorem 3.8). We prove that, under certain hypothesis, the handle element of a nearly Frobenius algebra $A$ belongs to its Jacobson radical (Proposition 3.10), but it does not belong to its socle generally (Example 3.9). It shows that Proposition 2.3.4 from \cite{A0} cannot be generalized to nearly Frobenius algebras.

In section 4 we introduce the notion of Schur element for nearly Frobenius algebras. We prove a generalization of Theorem 4.6 from \cite{B} where $\lambda: A \rightarrow B$ is an epimorphism of algebras between Frobenius algebras (Proposition 12), and we also prove that $(1) \Rightarrow (2)$ if $A$ is a symmetric nearly Frobenius algebra and $B$ is a symmetric algebra. However, the converse is not true (Example 4.2). 

\section{Preliminaries}
\subsection{Path algebras}

If $Q = (Q_0,Q_1,\st,\tg)$ is a finite connected quiver, $\Bbbk Q$ denotes its associated path algebra. Given $\rho$ a path in $\Bbbk Q$, we denote by $\len(\rho)$, $\st(\rho)$ and $\tg(\rho )$ the length, start and target of $\rho$, respectively. We denote by $J$ the ideal generated by all the arrows from $Q_1$. 
Let $x \in Q_0$, we say that $x$ is a sink (source) if there are no arrows $\rho$ such that $\st(\rho) = x$ ($\tg(\rho) = x$). \\
We say that an ideal $I \subset \Bbbk Q$ is admissible if there exist $k \geq 2$ such that $J^k \subset I \subset J^2$.\\

In the following definitions consider $I$ an admisible ideal and $A = \frac{\Bbbk Q}{I}$. 
\begin{defn}
 $A$ is a \emph{radical square zero} algebra if $ I = J^2$.

\end{defn}
\begin{defn}
	The algebra $A$ is a \emph{toupie} algebra if $Q$ has one sink $(\omega)$ and one source $(0)$, and for every vertex $x$ different to $\omega$ and $0$ there is only one arrow $\alpha$ such that $\st(\alpha)=x$ and only one arrow $\beta$ such that $\tg(\beta)=x$.
\end{defn}

\subsection{Frobenius and nearly Frobenius algebras}
In what follows $R$ is an associative commutative ring with unit.

\begin{defn}
		An associative $R$-algebra $A$ is \emph{Frobenius} if $A$ is finitely generated and projective as a $R$-module, and there exists a left $A$-module isomorphism 
	$$\lambda_l:A\rightarrow A^*.$$
\end{defn}
\begin{rem}
	The isomorphism $\lambda_l$ induces a morphism $\varepsilon_A:A\to R$ as $\varepsilon_A=\lambda_l\bigl(1_A\bigr),$ which is called \emph{trace}.
\end{rem}

\begin{defn}
An associative $R$-algebra $A$ is \emph{symmetric} if $A$ is a Frobenius algebra with trace $\varepsilon_A:A\to R$ such that $\varepsilon_A(ab)=\varepsilon_A(ba)$ for all $a,b\in A$.
\end{defn}
 Another characterization of Frobenius algebra, given by Abrams for example in \cite{A1}, is that an algebra $A$ is Frobenius if it has a coassociative counital comultiplication $\Delta:A\rightarrow A\otimes_R A$ which is a map of $A$-bimodules.
 
 From this characterization, the following definition arises naturally:
\begin{prop}\label{proposition2.3}
	If $A$ is a symmetric algebra, then its coproduct verifies the following property
	$$\Delta\bigl(1_A\bigr)=\tau\circ\Delta\bigl(1_A\bigr),$$
	where $\tau$ is the transposition.  
\end{prop}
\begin{proof}
The coproduct is characterized by the relation 
$$\varepsilon\bigl(x(ab)\bigr)=\sum \varepsilon\bigl(x_1a\bigr)\varepsilon\bigl(x_2b\bigr),\; \mbox{where}\;\Delta(x)=\sum x_1\otimes x_2$$
then, using that $\varepsilon(ab)=\varepsilon(ba)$ for all $a,b\in A$, we have that $\Delta\bigl(1_A\bigr)=\tau\circ\Delta\bigl(1_A\bigr).$
\end{proof}

\begin{defn}
	An associative $R$-algebra $A$ is a \emph{nearly Frobenius algebra} if it admits an homomorphism $\Delta:A\rt A\ot_R A$ of $A$-bimodules.

	Let $(A,\Delta_A)$ and $(B,\Delta_B)$ be two nearly Frobenius algebras. An homomorphism $f:A\rightarrow B$ is a \emph{nearly Frobenius homomorphism} if it is a morphism of $R$-algebras and the following diagram commutes
	$$\xymatrix{A\ar[r]^{f}\ar[d]_{\Delta_A}& B\ar[d]^{\Delta_B}\\
		A\otimes_R A\ar[r]_{f\otimes f}&B\otimes_R B
	}.$$
\end{defn}

Based on Proposition \ref{proposition2.3} we give the following definition.
\begin{defn}
	A nearly Frobenius algebra $A$ is \emph{ symmetric} if the coproduct satisfies that 
		$$\Delta\bigl(1_A\bigr)=\tau\circ\Delta\bigl(1_A\bigr).$$
\end{defn}

\begin{example}
		We consider $(A, \Delta)$ the following nearly Frobenius algebra, $A$ is the truncated polynomial algebra in one variable $\Bbbk[x]/x^{n+1}$ and $$\de(1)=\sum_{i+j=n+1}x^i\otimes x^j.$$
		It is immediate to see that $(A, \Delta)$ is a symmetric nearly Frobenius algebra.
\end{example}

\begin{rem}
	The Frobenius space $\mathcal{E}_A$ (see \cite{AGL15} for definition) associated to the algebra $A$ admits structure of $A$-module, where the action is given as follows: 
	$$\begin{array}{ccc}
		A\otimes \mathcal{E}_A&\to&\mathcal{E}_A\\
		u\otimes\de&\mapsto& \left[
		\begin{array}{cccc}
		u\ast\de:&A&\to&A\otimes A\\
			     &x&\mapsto&\bigl(x\otimes u\bigr)\de(1)
		\end{array}\right]
	\end{array}$$

	Observe that $$x\cdot(u\ast\de)(1)=(x\otimes 1)(1\otimes u)\de(1)=(1\otimes u)(x\otimes 1)\de(1)=(1\otimes u)\de(1)(1\otimes x)=(u\ast\de)(1)\cdot x$$
	\end{rem}


\section{Handle element}

In the context of Frobenius algebras, in \cite{K}, is defined the handle element as follows.\\ If $A$ is a Frobenius algebra, introduce the \emph{handle operator} as the $R$-linear map $\omega: A\to A$ defined as the composite $m\circ\Delta$, where $m$ is the multiplication of $A$ and $\Delta$ is the coproduct. This operator is a right (and left) $A$-module homomorphism, and it is given by multiplication by a central element $\omega_{A,\Delta}$. This element is called the \emph{handle element} and $\omega_{A,\Delta}=m\circ\Delta(1)$.

\subsection{Constructing handle elements} 
\begin{defn}
		If $(A,\de)$ is a nearly Frobenius $R$-algebra we define its \emph{handle element} as:
				$$\omega_{A,\de}:=m\circ\de(1)\in A.$$
\end{defn}

\begin{rem}
	The handle element is central: since $\de$ is an $A$-bimodule morphism, if $\de(1)=\sum x_i\otimes y_i$ then $\sum xx_i\otimes y_i=\sum x_i\otimes y_ix$, for all $x\in A$. Therefore, applying the multiplication, $\sum xx_i y_i=\sum x_i y_ix$ and $x\left(\sum_i x_i y_i\right)=\left(\sum_i x_i y_i\right)x$ which means $x\omega_A=\omega_Ax$, for all $x\in A$ and  $\omega_A\in Z(A)$. 
\end{rem}

\begin{prop}
	The handle element respects product and tensor product structure, i.e.
	$$\omega_{A_1\times A_2}=(\omega_{A_1},\omega_{A_2}),$$
	$$\omega_{A_1\otimes A_2}=\omega_{A_1}\otimes\omega_{A_2}.$$ 
\end{prop}
\begin{proof}
If we consider $A = A_1 \times A_2$, by Remark 6 of \cite{AGM2}, we have that $\de(1)=\sum(a_1,0)\otimes (a_2,0)+\sum (0,b_1)\otimes (0,b_2)\in A\otimes A$. Then 
$$\omega_{A_1\times A_2}=\sum (a_1a_2,0)+\sum(0,b_1b_2)$$
By the other hand, $\omega_{A_1}=\sum a_1a_2$ and $\omega_{A_2}=\sum b_1b_2$. Therefore
$$\omega_{A_1\times A_2}=\left(\sum a_1a_2,0\right)+\left(0,\sum b_1b_2\right)=\bigl(\omega_{A_1},\omega_{A_2}\bigr).$$
Now, if we consider $A=A_1 \otimes A_2$, again by Remark 7 of \cite{AGM2}, we have that 
$$\de_A:=\tau\circ \bigl(\de_{A_1}\otimes\de_{A_2}\bigr) :A_1\otimes A_2\rightarrow (A_1\otimes A_2)\otimes (A_1\otimes A_2)$$ where $\tau: (A_1\otimes A_1)\otimes(A_2\otimes A_2)\rightarrow (A_1\otimes A_2)\otimes (A_1\otimes A_2)$ it the transposition map. Then

$$\begin{array}{rcl}
\omega_A=m_A\circ\de_A&=&\bigl(m_{A_1}\otimes m_{A_2}\bigr)\circ \tau^{-1}\circ \tau\circ \bigl(\de_{A_1}\otimes\de_{A_2}\bigr)\\
        &=&\bigl(m_{A_1}\otimes m_{A_2}\bigr)\circ\bigl(\de_{A_1}\otimes\de_{A_2}\bigr)\\
        &=&\bigl(m_{A_1}\circ\de_{A_1}\bigr)\otimes\bigl(m_{A_2}\circ\de_{A_2}\bigr)\\
        &=&\omega_{A_1}\otimes\omega_{A_2}.
\end{array}$$
\end{proof}

\begin{prop}
		If $(A,\de)$ is a nearly Frobenius $R$-algebra and $\omega$ is its handle element, then for all $u\in Z(A)$, $u\omega$ is the handle element of $(A,u\ast\de)$.
	\end{prop}
\begin{proof}
	If $\de(1)=\sum x_1\otimes x_2$, then $u\ast\de(1)=\sum x_1\otimes ux_2$. Therefore $$\omega_{u\ast\de}=\sum x_1ux_2=\sum ux_1x_2=u\sum x_1x_2=u\omega.$$
\end{proof}

\begin{cor}
	If $(A,\de)$ is a nearly Frobenius commutative $R$-algebra and $\omega$ is its handle element, then for all $u\in A$, $u\omega$ is the handle element of the nearly Frobenius algebra $(A,u\ast\de)$.
\end{cor}
The next result appears in \cite{CIM} in the context of  Frobenius algebras.   
\begin{thm}\label{theo1}
	Let $(A, \de)$ be a nearly Frobenius $R$-algebra. If its handle element is a unit then $A$ is separable.
\end{thm}
\begin{proof}
If $\de(1)=\sum x_i\otimes y_i$, then $\omega_A=\sum x_iy_i$. As $\omega_A\in Z(A)$ we have that $\omega_A^{-1}\in Z(A)$. Then, defining $e=\omega_A^{-1}\sum_ix_i\otimes y_i$ we obtain that:
\begin{enumerate}
	\item $x\cdot e=x\omega_A^{-1}\sum_ix_i\otimes y_i=\omega_A^{-1}x\sum_ix_i\otimes y_i=\omega_A^{-1}\sum_ix_i\otimes y_ix=e\cdot x.$
	\item $m(e)=m\left(\omega_A^{-1}\sum_ix_i\otimes y_i\right)=\omega_A^{-1}\sum x_iy_i=\omega_A^{-1}\omega_A=1$.
\end{enumerate}
Then, by Proposition 19 of \cite{AGM2}, we conclude that $A$ is separable.
\end{proof}

Some examples of applications of Theorem \ref{theo1} are given below.

\begin{example}\label{matrix}
	Let be $A$ the matrix algebra $M_{n\times n}(\Bbbk)$. We consider the canonical basis of $A$, $\mathcal{B}=\bigl\{E_{ij}:\; i,j=1,\dots,n\bigr\}$.
	Then $\mathcal{A}=\bigl\{\de_{kj}:\; k,l=1,\dots,n\bigr\}$ is a basis of $\mathcal{E}_A$ where
	$$\de_{kl}(E_{ij})=E_{ik}\otimes E_{lj}, \;\forall k,l=1,\dots,n.$$ 
	For this family of coproducts their corresponding handle elements are
	$$\omega_{kl}=m\circ\de_{kl}(I)=m\circ\de_{kl}\Bigl(\sum_{i} E_{ii}\Bigr)=\sum_{i}m\circ \de_{kl}\bigl(E_{ii}\bigr)=\sum_{i}m\bigl(E_{ik}\otimes E_{li}\bigr)= \begin{cases} 0 & \text{ if } l\neq k, \\ I & \text{ if } k=l. \end{cases}$$
	Then $\omega_{kk} =I$, $\forall k=1,\dots,n$ and $\de_{kk}$ is a normalized nearly Frobenius coproduct (see Definition 14 of \cite{AGM2}). \\
	The coproduct $\displaystyle{\de_0=\sum_{i=1}^n \de_{ii}}$ corresponds to the Frobenius structure of $M_{n\times n}(\Bbbk)$ and $$\omega_0=m\circ\de_0(I)=\sum_{i,k=1}^nE_{ik}E_{ki}=\sum_{i,k=1}^nE_{ii}=\sum_{k=1}^nI=nI.$$
In particular, by Theorem \ref{theo1},  we conclude that $A$ is a separable algebra.
	
\end{example}

\begin{example}
	Let $G$ be a cyclic finite group of order $n$ and the group algebra $\k G$, with the natural basis $\bigl\{g^i:\;i=1,\dots,n\bigr\}$. A basis of the Frobenius space is
	$$\ma{B}=\bigl\{\de_k:\k G\rt \k G\ot \k G: k\in{1,\dots, n}\bigr\},$$
	where $\displaystyle{\de_1(1)=\sum_{i=1}^ng^i\ot g^{n+1-i}}$ and $\displaystyle{\de_k\bigl(1\bigr)=\sum_{i=1}^{k-1}g^i\ot g^{k-i}+\sum_{i=k}^ng^i\ot g^{n+k-i}}$ for $k=2,\dots, n$. In this case we have that $$\omega_k=ng^k, \;\forall k=1,\dots, n.$$
	Again, we can conclude that this algebra is separable.
\end{example}

\begin{rem}
 
The converse of Theorem \ref{theo1} is not fulfilled, consider $\de_{kl}$ with $k\neq l$ in Example \ref{matrix} as a counterexample. Even if we asked $\omega\neq 0$, the converse of Theorem \ref{theo1} would still fail, just look at the following example to check it.
\end{rem}

\begin{example}
	Let $R$ be a commutative ring and $A=R\oplus R$ be the ring direct sum of two copies of $R$. Let $e_1 = (1, 0)$ and $e_2 = (0, 1)$ be the orthogonal idempotents in $A$. We prove that $A$ is separable over $R$. For this we consider the element $e=e_1\otimes e_1+e_2\otimes e_2\in A\otimes A$, it is a separability element for $A$:
$$m(e)=m\bigl(e_1\otimes e_1+e_2\otimes e_2\bigr)=m\bigl(e_1\otimes e_1\bigr)+m\bigl(e_2\otimes e_2\bigr)=e_1+e_2=1_A.$$
As $R$ is a commutative ring and the tensor product is over $R$ it is clear that $r e=e r$, for all $r\in R$.
Therefore $A$ is separable over $R$.\\
On the other hand we study the nearly Frobenius structures of $A$. It is not difficult to prove that a general nearly Frobenius coproduct has the expression
$$\de(1)=ae_1\ot e_1+be_2\ot e_2,\;\forall a,b\in R.$$
In particular the Frobenius structure is $\de_0(1)=e_1\ot e_1+e_2\ot e_2$, and $\varepsilon(e_1)=\varepsilon(e_2)=1$.\\
The handle element associated to the Frobenius structure is $$\omega_{\de_0}=m\bigl(\de_0(1)\bigr)=e_1+e_2=1_A.$$
The handle element $\omega_{\de_0}$ is a unit of $A$. But, if we fix the nearly Frobenius structure $\de_1(1)=e_1\ot e_1$ on $A$, then the handle element associated $\omega_{\de_1}=e_1$ is not a unit of $A$. 
 \end{example}

If we adapt Theorem \ref{theo1} we can get an equivalence as follows:

\begin{thm}\label{teo10}
An algebra $A$ admits a nearly Frobenius structure where the handle element is a unit of $A$ if and only if $A$ is separable.  
\end{thm}	
\begin{proof}
	$(\Rightarrow)$ It is Theorem \ref{theo1}.\\
	$(\Leftarrow)$ Applying Theorem 22 of \cite{AGM3} we have that $A$ admits a normalized nearly Frobenius coproduct $\de$, that is $m\circ\de=\operatorname{Id}_A$. Then $\omega_\de=m\bigl(\de(1)\bigr)=1\in A$. Therefore $\omega_\de$ is a unit of $A$.
\end{proof}

\begin{example}
	Let $A$ be the truncated polynomial algebra in one variable $\Bbbk[x]/x^{n+1}$. In this case, a  basis of the Frobenius space is $\ma{B}=\{\de_0,\dots,\de_n\}$, where $$\de_k(1)=\sum_{i+j=n+k}x^i\otimes x^j,\;\forall k\in\{0,1,\dots,n\}.$$ We can determine the handle elements associated to the nearly Frobenius coproducts described.
	$$\omega_k=\sum_{i+j=n+k}x^{i+j}=0,\;\forall k=1,\dots,n$$
	and $$\omega_0=\sum_{i+j=n}x^{i+j}=(n+1)x^n.$$
	Note that the only coproduct that admits a non zero handle element is the Frobenius coproduct.   Using Theorem \ref{teo10} we can conclude that if $n>0$ $A$ is not separable.
\end{example}
\vspace{0.3cm}

\subsection{Handle element of algebras over fields}
We will now study some properties of the handle element in the particular case of $\k$-algebras, where $\k$ is a field.

\begin{prop}\label{prop1}
	Let be $A$ a nearly Frobenius algebra of finite dimension over $\k$ and $\omega$ its corresponding handle element. If $\omega$ is not a unit, then there exist $k\in\mathbb{N}$ such that $\omega^k\in\operatorname{soc}(A)$. 	
\end{prop}
\begin{proof}
	If $\de$ is a nearly Frobenius coproduct in $A$ it is an $A$-bimodule morphism. Then, if $A$ has finite dimension $\de(J)\subset J\ot J$, where $J$ is the Jacobson radical of $A$.\\ If we define $\psi=m\circ\de|_J:J\rt J^2$ we obtain that $\psi(x)=x\omega$. Composing the function $\psi$ with itself $m$ times we obtain that $\psi^m:J\rt J^{2m}$ is $\psi^m(x)=x\omega^m$.\\
	As $A$ is a finite dimensional algebra, there exists $k$ such that $J^{2k}=\{0\}$, in particular the map $\psi^k=0$. Then $\psi^k(x)=x\omega^k=0$\; $\forall x\in J$.\\
	As $\omega$ is not a unit we have that $$J\omega^k=\{0\}\Rightarrow \omega^k\in\operatorname{soc}(A).$$
\end{proof}

\begin{prop}\label{prop3}
	Let be $A$ a connected finite dimensional nearly Frobenius  path algebra  which is not a single vertex over $\k$ ($\k$ algebraically closed) and $\omega$ its corresponding handle element. Then $\omega\in J$, where $J$ is the Jacobson radical of $A$. 	
\end{prop}
\begin{proof}
Consider the path of length zero associated to the vertex $i$, $e_i$. Then $\Delta(e_i)=\sum_{j}a_j\otimes b_j$ with $s(b_j)=t(a_j)=e_i$. Suppose $e_i\otimes e_i$ is a summand in $\Delta (e_i)$. Since $A$ is connected there is at least one arrow that starts or ends in $i$. Let us suppose that $\alpha$ starts on $i$ and ends in $k$ (the other case is analogous). Then $\Delta(\alpha)=\Delta(e_i)\alpha=\alpha\Delta(e_k)$ so $e_i\otimes \alpha$ is a summand in $\Delta(\alpha)$ but cannot be of the form $\alpha v\otimes u$ with $u,v$ some paths and this is absurd. Then the elements of the form $e_i\otimes e_i$ do not appear in $\Delta(e_i)$ and so do not appear in $\Delta(1)=\sum_i\Delta(e_i)$ which implies that every summand of $\omega$ is a path of positive lenght and then $\omega\in J$.  
\end{proof}

\begin{rem}
In the hypothesis of the above proposition we conclude that $\omega$ is nilpotent.
\end{rem}

\subsection{Families of nearly Frobenius algebras with zero handle element}
In this section we describe some families of algebras that admit nontrivial nearly Frobenius structures but whose corresponding handle elements are null.

\begin{example}
 A family of algebras where we can find its handle element was presented in 5.3 of \cite{AGL15}.\\
 If $\displaystyle{A=\frac{\Bbbk C}{I_C}}$, where $C=C\bigl(n_1,n_2,\dots, n_m\bigr)$, with $m>1$ and $I_C=\bigl\langle\alpha_{n_m}^{m}\alpha_1^1, \alpha_{n_i}\alpha_{1}^{i+1}, i=1,\dots, m-1\bigr\rangle$,
$$\scalebox{0.6}{\includegraphics{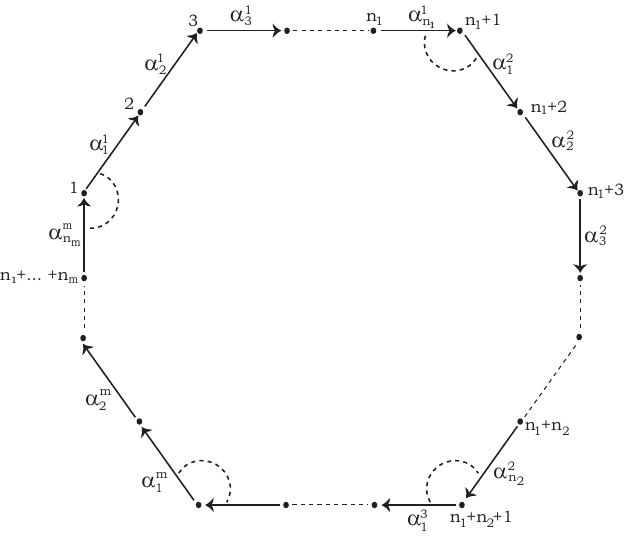}}$$
by the description of  nearly coproduct in Theorem 5 of \cite{AGL15} we conclude that $\omega=0$.\\

	 
\end{example}

\begin{prop} \label{sinlazos}
Let $A = \frac{\Bbbk Q}{I}$ be an finite dimensional algebra. If every cycle of $Q$ is $0$ in $A$, then every handle element of $A$ is $0$. 
\begin{proof}
Suppose there is a coproduct $\de$ on $A$, such that $\omega_\de=m\circ\de (1) \not = 0$. Then, there is a vertex $p \in Q_0$ such that $m\circ\de (e_p) \not = 0$. Note that for every nearly Frobenius coproduct $\de$ and vertex $p \in Q_0$,  $\de (e_p) = \sum a_i\alpha_i \ot \beta_i$ where $\tg(\beta_i) = \st(\alpha_i) = p, \forall i$. So we have an index $i$ such that $\alpha_i \beta_i \not = 0$ in $A$, in particular $\st(\beta_i) = \tg(\alpha_i)$. Then $\alpha_i \beta_i $ is a non zero cycle. 
\end{proof}
\end{prop}

The following results are obtained from the previous proposition.

\begin{cor}
	If $A$ a toupie algebra over a field $\Bbbk$ then every handle element is $0$.
\begin{proof}
Since toupie algebras are acyclic the result follows immediately from Proposition \ref{sinlazos}.
\end{proof}	
\end{cor}

\begin{cor}
	If $A$ is a radical square zero algebra different to the loop then every handle element is $0$. 
\begin{proof}
If every cycle has length bigger to or equal to $2$ the result follows from Proposition \ref{sinlazos}.  By Remark 2 of \cite{AGM1} if the indegree and the outdegree of a vertex $p$ are both bigger to or equal to $2$ we have that $m \circ \Delta (e_p) = 0$, so we can suppose that for a vertex $p$ there is a loop $\alpha$ at $p$ and arrows $\beta_i$ that starts at $p$. Again by Remark 2 of \cite{AGM1} $\Delta (e_p) =  \sum b_i \beta_i \otimes \alpha$. Thus $m \circ \Delta (e_p) = \sum b_i \beta_i \alpha = 0$ since $A$ is a radical square zero algebra.     
\end{proof}	
\end{cor}


\subsection{Handle element and socle}
In the context of commutative Frobenius algebras, Theorem \ref{teo9} is valid, which allows to determine the socle of the algebra from its handle element. In this section we answer the question, is it possible to generalize this result to non commutative algebras or nearly Frobenius algebras?
\begin{thm}[Proposition 3.3 of \cite{A}]\label{teo9}
	In a commutative  Frobenius algebra $A$ with handle element $\omega$, the ideal $\omega A$ is the socle of $A$.
\end{thm} 

\begin{rem}
	The next example shows that the Theorem \ref{teo9} is not valid in the case of non commutative Frobenius algebras. 
\end{rem}

\begin{example}\label{example1}
	Let be $A=\Bbbk Q/I$, where $Q=\xymatrix{
		\bullet_1\ar@/^/[r]^\alpha&\bullet_2\ar@/^/[l]^\beta}$ and $I=\langle\alpha\beta,\beta\alpha\rangle$. In this example the nearly Frobenius structures that it admits are the following
	$$\begin{array}{rcl}
		\de(e_1)&=&ae_1\otimes\beta+b\alpha\otimes e_1+c\alpha\otimes\beta\\
		\de(e_2)&=&be_2\otimes \alpha+a\beta\otimes e_2+d\beta\otimes\alpha
	\end{array}	$$ where $a,b,c,d\in\Bbbk.$
	
	A Frobenius coproduct is
	$$\de(1)=e_1\otimes\beta+\alpha\otimes e_1+e_2\otimes\alpha+\beta\otimes e_2-\alpha\otimes\beta-\beta\otimes\alpha,$$
	where the trace is defined as
	$$\varepsilon:A\rt\Bbbk,\;\varepsilon(e_1)=\varepsilon(e_2)=\varepsilon(\alpha)=\varepsilon(\beta)=1.$$
	In this case, $\operatorname{soc}(A)=\alpha A\oplus\beta A$, and $\omega=m\circ\de(1)=0.$\\ Then, the handle element does not generate the socle.
\end{example}

Given a finite dimensional algebra $A$ and its Jacobson radical $J$, for a left ideal $M \subset A$, $J M= 0 $ if and only if $M \subset \operatorname{soc}(A)$. As a consequence, if $\operatorname{soc}(A) = S_1 \subset S_2 \subset\cdots\subset S_m = A$ is a socle series, then $J S_i \subset S_{i-1}$.\\ Recall that for a Frobenius algebra $A$ we have that $\operatorname{soc}(A_A) = \operatorname{soc}( _AA)$.\\
 
In the context of non commutative Frobenius algebras we have the following result, which is a weakening of Theorem \ref{teo9}.

\begin{thm}\label{teo 16}
In a non-commutative Frobenius algebra $A$ with handle element $\omega$, the ideal $\omega A$ is a submodule of the socle of $A$.
\end{thm}
\begin{proof}
Let $\operatorname{soc}(A) = S_1 \subset S_2 \subset\cdots \subset S_m = A$ be a socle series for $A$ as a left $A$-module. Choose a basis for $S_1$. Now, starting with $i =1$, iteratively take the basis for $S_i$ and extend it to a basis for $S_{i+1}$. Denote the elements of the basis for $S_m = A$ by $e_1,\cdots , e_n$. Suppose that  $e_i \in S_k \setminus S_{k-1}$ and $x \in J$, then $xe_i \in S_{k-1}$ and $\displaystyle{ xe_i = \sum_{j\neq i}^n \alpha_je_j}$. \\
Remember that $A$ is a Frobenius algebra if it possesses a left $A$-module isomorphism $\lambda_l:A\rt A^*$ with its dual vector space, where $A$ is viewed as the left module over itself, and $A^*$ is made a left
$A$-module by the action $(a\cdot f)(b)=f(ba)$, for all $a\in A$ and $f\in A^*$.
In this context we can construct the counit $\varepsilon:A\rt \k$ of $A$ as  $\varepsilon(a)=\lambda_l(1)(a)$, for all $a\in A$. In particular,  
 we have $e_1^\#,\dots, e_n^\#$ the corresponding dual basis elements of $A$, which verify $\varepsilon\bigl(e_ie_j^\#\bigl)=\delta_{ij}$ for all $i,j=1,\dots,n$. Then 
$$\varepsilon\bigl(xe_ie_i^\#\bigr)=\varepsilon\left(\sum_{j\neq i}\alpha_j e_je_i^\#\right)=\sum_{j\neq i}\alpha_j \varepsilon\bigl(e_je_i^\#\bigr)=0.$$
Therefore $Je_ie_i^\#\subset\operatorname{Ker}(\varepsilon)$. But, $\operatorname{Ker}(\varepsilon)$ can contain no nontrivial left ideals, then $$Je_ie_i^\#=\{0\}.$$ Thus $e_ie_i^\#\in \operatorname{soc}(A)$ for each $i$, so $\displaystyle{\omega=\sum_{i=1}^ne_ie_i^\#\in\operatorname{soc}(A)}.$

\end{proof}

The following examples show that Theorem \ref{teo 16} does not hold in the context of nearly Frobenius algebras. \\



The first two examples illustrate that there are nearly Frobenius (non Frobenius) algebras where the ideal generated by the handle element is strict contained in the socle of the algebra.
\begin{example} 
	We consider a small modification of the Example \ref{example1}, let be $A=\Bbbk Q/I$, where $Q=\xymatrix{
	\bullet_1\ar@/^/[r]^\alpha&\bullet_2\ar@/^/[l]^\beta}$ and $I=\langle\beta\alpha\rangle$. In this case we can prove that the nearly Frobenius structures are the following
$$\begin{array}{rcl}
\de(e_1)&=&a_1(e_1\otimes\alpha\beta+\alpha\beta\otimes e_1)+a_2\alpha\otimes\beta+a_3\alpha\otimes\alpha\beta+a_4\alpha\beta\otimes\beta+a_5\alpha\beta\otimes\alpha\beta\\
\de(e_2)&=&a_1\beta\otimes\alpha
\end{array}
$$ for $a_1,\dots,a_5\in\Bbbk$.\\ 
A first observation we can make is that this algebra does not support Frobenius algebra structure, because if we suppose that there exist $\varepsilon:A\rt\Bbbk$ trace of $\de$, then
$$(\varepsilon\otimes 1)\de(e_2)=a_1\varepsilon(\beta)\alpha=e_2,$$ and this is a contradiction. \\
On the other hand, we can determine the handle element for $(A,\de)$,
$$\omega=m\circ\de(1)=m\circ\de(e_1+e_2)=2a_1\alpha\beta+a_2\alpha\beta=(2a_1+a_2)\alpha\beta,$$
and the socle of $A$ is
$$\operatorname{soc}(A)=\alpha\beta A\oplus\beta A,$$ then $\omega A\subsetneq \operatorname{soc}(A).$
	\end{example}

\begin{example}[Local rings which are not Gorenstein, see \cite{K}]
	In $A = \Bbbk[x, y]/\bigl(x^2, xy^2, y^3\bigr)$, the socle is $\bigl(xy, y^2\bigr)$. This ring does not support Frobenius algebra structure but it does admit nearly Frobenius algebra structure.
	A general nearly Frobenius coproduct is
	$$\de(1)=a_1\bigl(x\ot y^2+y\ot xy+y^2\ot x+xy\ot y\bigr)+a_2\bigl(x\ot xy+xy\ot x\bigr)+a_3y^2\ot y^2+a_4y^2\ot xy$$$$+a_5xy\ot y^2+a_6xy\ot xy,$$
	where $a_i\in\Bbbk,\forall i=1,\dots, 6.$ \\
	Note that there are no scalars that make any coproduct admit a trace, therefore none can be a Frobenius coproduct.\\ On the other hand, the handle element is zero for all coproduct
	$$\omega_\de=0,\;\forall a_i\in\Bbbk, i=1,\dots,6 \Rightarrow \omega_\de A=\{0\}\subsetneq\operatorname{Soc}(A)=\bigl(xy, y^2\bigr).$$

	In the ring $A =\Bbbk[x, y, z]/\bigl(x^2, y^2, z^2, xy\bigr)$, the socle is $(xz, yz)$. As before,  this ring does not support Frobenius algebra structure but it does admit nearly Frobenius algebra structure, where the general Frobenius coproduct is generated, as a linear combination, by 
	$$\de_1(1)=x\ot xz+xz\ot x,\; \de_2(1)=x\ot yz+xz\ot y,\; \de_3(1)=y\ot xz+yz\ot x,$$ 
	$$\de_4(1)=y\ot yz+yz\ot y,\; \de_5(1)=xz\ot xz,\;\de_6(1)=xz\ot yz,$$
	$$\de_7(1)=yz\ot xz,\;\de_8(1)=yz\ot yz.$$
	In particular $\operatorname{Frobdim}A=8$, and none of the nearly Frobenius coproducts can be completed to a Frobenius coproduct. Also, all the handle elements are zero ($\omega_\de=0$). Then $\{0\}=\omega_\de A\subsetneq\operatorname{Soc}(A)=\bigl(xz, yz\bigr).$ 
\end{example}

The following example shows that Theorem \ref{teo 16} does not generalize to nearly Frobenius algebras.
\begin{example}
Let be $A=\Bbbk Q/I$, where $Q$ is $$\xymatrix{
	&\bullet_2\ar[dr]^\beta&\\
	\bullet_1\ar[ur]^\alpha&&\bullet_3\ar[ll]^{\gamma}}$$ and $I=\langle\alpha\beta\gamma\rangle$. In this example we can describe all the nearly Frobenius structures and as consequence determine the handle element 
$$\omega=a\bigl(\beta\gamma\alpha+\gamma\alpha\beta\bigr),\quad\mbox{where}\;a\in\Bbbk.$$
On the other hand, the socle of $A$ is $$\operatorname{soc}(A_A)=\gamma\alpha\beta A\oplus \beta\gamma\alpha\beta A.$$
Then $\omega A\not\subset\operatorname{soc}(A_A)$.
	\end{example}




\section{Casimir and Schur elements}

Given an algebra $A$, there is a natural isomorphism $\psi: A\otimes A^*\to \operatorname{End}_\Bbbk(A)$ such that $\psi(x\otimes \beta)(y)=\beta(x)y$. If $A$ is Frobenius then $A\cong A^*$ and we can consider an isomorphism $\tilde{\psi}:A\otimes A\to \operatorname{End}_\Bbbk(A)$.
The \emph{Casimir element} (see \cite{JL}) for $(A, \lambda)$ be a symmetric algebra is defined as the element $C_A\in A\otimes A$ that corresponds to $\operatorname{Id_A}\in \operatorname{End}_\Bbbk(A)$ under the isomorphism $\operatorname{End}_\Bbbk(A)\cong A\otimes A$.\\
If $\bigl\{e_i\bigr\}$ and $\bigl\{e^i\bigr\}$ are dual basis with respect to $\lambda$, the Casimir element is:
$$C_A=\sum e_i\otimes e^i=\sum e^i\otimes e_i.$$ 
This element satisfies that
$$\sum ae_i\otimes e^i=\sum e_i\otimes e^ia,\quad\mbox{for all}\; a\in A.$$
The \emph{central Casimir element} of $A$ is defined as  
$$Z_{A}=\sum e^ie_i.$$
\begin{rem}
	Note that the Casimir element of $A$ coincide with $\Delta(1)$, if $\Delta$ is the coproduct of this symmetric algebra (see \cite{A1}), and the central Casimir element coincide with the handle element.
\end{rem}
In the context of Frobenius algebras we can define the Casimir element in a similar way.\\
Given a basis $\{e_1,\dots,e_n\}$ of the Frobenius algebra $A$ with isomorphism $\lambda_l: A \to A^*$, let $\bigl\{e^1,\dots, e^n\bigr\}$ be the \emph{dual basis} of $A$ relative to $\lambda_l$, that means the elements satisfy $\lambda_l\bigl(e^j\bigr)(e_i)=\delta_{ij}$, for all $i,j=1,\dots,n$. In particular, if $\{e_1^*,\dots, e_n^*\}$ denotes the basis of $A^*$ satisfying $e_i^*(e_j)=\delta_{ij}$, using $\lambda_l:A\rt A^* $,  we conclude that $e^j=\lambda_l^{-1}(e_j^*)$ for all $j=1,\dots, n$.\\
In this context, if we use the isomorphism $\lambda_l: A \to A^*$ as left $A$-modules, the \emph{Casimir element}, as before, corresponds to $\operatorname{Id_A}\in \operatorname{End}_\Bbbk(A)$ under the isomorphism $\operatorname{End}_\Bbbk(A)\cong A\otimes A$. But, in this case we only have the following expression since $A$ is not necessarily symmetric:
$$C_A=\sum_{i=1}^n e^i\otimes e_i.$$

Let us recall the Proposition 2 proved in \cite{AGM3}.
\begin{prop}
	The coproduct $\Delta:A\rightarrow A\otimes A$ of the Frobenius algebra $A$ satisfies
	\begin{enumerate}
		\item[(i)] $\displaystyle{\de(1_A)=\sum_{i=1}^ne^i\otimes e_i}$. 
		\item[(ii)] $\displaystyle{\de(x)=\sum_{i=1}^nxe^i\otimes e_i=\sum_{i=1}^ne^i\otimes e_ix}$, for all $x\in A$.
	\end{enumerate}  
\end{prop}

Then in this case, as before, the Casimir element coincides with $\Delta(1)$.

\begin{defn}
	If $A$ is a nearly Frobenius algebra, we define the \emph{Casimir element} of $A$ as $$C	_A=\Delta(1).$$
\end{defn} 
This allows us to view the properties of the Casimir element or related constructions in terms of the coproduct of the (nearly) Frobenius algebra.\\

Let $(A,\varepsilon_A)$ and $(B,\varepsilon_B)$ be Frobenius algebras and $\phi:A\twoheadrightarrow B$ an epimorphism of algebras. Note that $\varepsilon_B\circ\phi$ is in $A^{\ast}$, so there exists an element $u_0^B\in A$ such that $u_0^B\cdot\varepsilon_A=\varepsilon_B\circ\phi$. Consider $s_{\phi}=\phi\bigl(u_0^B\bigr)\in B$.

\begin{defn}
The element $s_\phi$ is called the \emph{Schur element} of the epimorphism $\phi:A\twoheadrightarrow B$.
\end{defn}

\begin{rem}
The Schur element can be describe as
$$s_\phi=\sum_{i=1}^n\varepsilon_B\bigl(\phi\bigl(e^i\bigr)\bigr)\phi\bigl(e_i\bigr):$$
Let $\displaystyle{\Delta_A(1)=\sum_{i=1}^ne^i\otimes e_i}$ with $\displaystyle{\Delta_A(a)=\sum_{i=1}^nae^i\otimes e_i=\sum_{i=1}^ne^i\otimes e_ia}\;$ for all $a\in A$.\\
As $u_0^B\in A$ we have that 
$$u_0^B=\bigl(\varepsilon_A\otimes \operatorname{Id}_A\bigr)\Delta_A\bigl(u_0^B\bigr)=\sum_{i=1}^n\varepsilon_A\bigl(u_0^Be^i\bigr)e_i=\sum_{i=1}^n\bigl(u_0^B\cdot\varepsilon_A\bigr)\bigl(e^i\bigr)e_i=\sum_{i=1}^n\bigl(\varepsilon_B\circ\phi\bigr)\bigl(e^i\bigr)e_i $$
Then
$$s_\phi=\phi\bigl(u_0^B\bigr)=\sum_{i=1}^n\bigl(\varepsilon_B\circ\phi\bigr)\bigl(e^i\bigr)\phi\bigl(e_i\bigr)=\sum_{i=1}^n\varepsilon_B\bigl(\phi\bigl(e^i\bigr)\bigr)\phi\bigl(e_i\bigr).$$
\end{rem}
The previous remark allows us to generalize the definition of Schur element in the context of nearly Frobenius algebras. We will denote $\displaystyle{\Delta(1)=\sum_{i=1}^{n} e^{i}\otimes e_{i}}$ in the nearly Frobenius case although $\bigl\{e_1, e_2, \cdots e_n\bigr\}$ and $\bigl\{e^1, e^2, \cdots e^n\bigr\}$ are not dual basis.
 
\begin{defn}
		Let $A$ be a nearly Frobenius algebra and $B$ a Frobenius algebra with $\phi: A\twoheadrightarrow B$ an epimorphism of algebras. We define the \emph{Schur element} of $\phi$, $s_\phi\in B$,  as the element
		$$s_\phi=\sum_{i=1}^n\varepsilon_B\bigl(\phi\bigl(e^i\bigr)\bigr)\phi\bigl(e_i\bigr).$$
		
\end{defn}

\begin{prop}
	Let $A$ be a nearly Frobenius algebra and $B$ a Frobenius algebra with $\phi: A\twoheadrightarrow B$ an epimorphism of algebras. Then:
	\begin{enumerate}
		\item $(\phi\otimes \phi)(C_A)=s_{\phi}C_B$,
		\item $\phi(Z_A)=s_{\phi}Z_B$,
	\end{enumerate}
	with $C_A=\sum_{i=1}^n e^i\otimes e_i$ the Casimir element of $A$ and $Z_{A}=\sum_{i=1}^n e^ie_i$ the central Casimir element of $A$. \end{prop}
\begin{proof}
Consider the map  $\psi:B\otimes_{R} B\to End_{R}(B)$ as in the definition of the Casimir element:
	given $x\otimes y\in B\otimes_{R} B$,  $\psi(x\otimes y)$ is the map in $End_{R}(B)$ that sends $b\in B$ to $\varepsilon_B(bx)y$. 
	If $C_B=\sum f^i\otimes f_i$ then
	$$\psi(C_B)(b)=\sum \varepsilon_B(bf^i)f_i=b$$
	and  $\psi(C_B)=Id_B.$ \\
	Let us now prove that $ \psi(\phi\otimes \phi)(C_A)=s_{\phi}Id_B$:
	\begin{align*}
		\psi(\phi\otimes \phi)(C_A)(b)&=\psi\Bigl(\sum \phi(e^i)\otimes \phi(e_i)\Bigr)(b)\\&=\sum\varepsilon_B(b\phi(e^i))\phi(e_i)\\&=\sum\varepsilon_B(\phi(a)\phi(e^i))\phi(e_i)\\&=\sum\varepsilon_B(\phi(ae^i))\phi(e_i) \end{align*}
	Since $A$ is a nearly Frobenius algebra $$\Delta(a)=\sum ae^i\otimes e_i=\sum e^i\otimes e_ia$$ and $$(\phi\otimes \phi)\Bigl(\sum ae^i\otimes e_i\Bigr)=\sum \phi(ae^i)\otimes \phi(e_i)=\sum \phi(e^i)\otimes \phi(e_ia).$$ Then applying $\varepsilon_B \otimes 1$ we obtain $$\sum\varepsilon_B(\phi(e^i))\phi(e_ia)=\sum\varepsilon_B(\phi(ae^i))\phi(e_i)$$ and $$ \psi(\phi\otimes \phi)(C_A)(b)=s_{\phi}b.$$\bigbreak
	To prove \emph{(2)} observe that
	$$\phi(Z_A)=\phi\Bigl(\sum e^ie_i\Bigr)=\sum\phi(e^i)\phi(e_i)=m\circ (\phi\otimes \phi)(C_A)=m\circ(s_{\phi}C_B)=s_{\phi}Z_B.$$ 
\end{proof}

	As the central Casimir element of $A$ is the handle element of $A$ the previous proposition  can be written as follows.

\begin{prop}
		Let $A$ be a nearly Frobenius algebra and $B$ a Frobenius algebra with $\phi: A\twoheadrightarrow B$ an epimorphism of algebras. Then:
	\begin{enumerate}
		\item $(\phi\otimes \phi)(\Delta_A(1_A))=s_{\phi}\Delta_B(1_B)$.
		\item $\phi(\omega_A)=s_{\phi}\omega_B$.
	\end{enumerate}
\end{prop}

\begin{lem}
Let $A$ be a nearly Frobenius algebra and $B$ a Frobenius algebra with $\phi: A\twoheadrightarrow B$ an epimorphism of algebras. Then the Schur element $s_\phi$  is central in $B$. 
\end{lem}
\begin{proof}
	Using that $\phi$ is an epimorphism it is enough to prove that $s_\phi\phi(a)=\phi(a)s_\phi$ for all $a\in A$.
$$\begin{array}{rclc}
	s_\phi\phi(a)&=& \displaystyle{\sum\varepsilon_B\bigl(\phi\bigl(e^i\bigr)\bigr)\phi\bigl(e_i\bigr)\phi(a)}&\\
	&=&\displaystyle{\sum\varepsilon_B\bigl(\phi\bigl(e^i\bigr)\bigr)\phi\bigl(e_ia\bigr)}&\mbox{as $\phi$ is a morphism of algebras}\\
	&=&\displaystyle{\sum\varepsilon_B\bigl(\phi\bigl(ae^i\bigr)\bigr)\phi\bigl(e_i\bigr)}&\mbox{$\Delta$ is a morphism of $A$-bimodules}\\
	&=&\displaystyle{\sum\varepsilon_B\bigl(\phi(a)\phi\bigl(e^i\bigr)\bigr)\phi\bigl(e_i\bigr)}&\mbox{as $\phi$ is a morphism of algebras}\\
	&=&\displaystyle{\sum\varepsilon_B\bigl(\phi(a)f^j\bigr)f_js_\phi} &\mbox{Proposition 4.8}\\
	&=&\phi(a)s_\phi&\mbox{$\varepsilon_B$ trace of $B$}
\end{array}$$ 
	
\end{proof}

\begin{lem} Let $A$ be a Frobenius $R$-algebra and $X, X'$ right $A$-modules such that $X$ is a finitely generated projective $R$-module then:\\
		The submodule $Hom^{proj}_{A}(X,X')$ of $Hom_{A}(X,X')$ consisting of maps that factorize through a finitely generated projective right $A$-module coincides with $Im(Tr_A)$ where
		$$\begin{array}{cccc}
			Tr_{A}:&Hom_{R}(X,X')&\to& Hom_{A}(X,X')\\
			& \alpha&\mapsto&\Bigl[\displaystyle{x\mapsto \sum \alpha\bigl(xe^i\bigr)e_i}\Bigr]	
		\end{array}$$
		
\end{lem}
\begin{proof}
	We will first consider $\beta\in Hom^{proj}_{A}(X,X')$ and prove that $\beta\in Im(Tr_{A})$. Since $\beta\in Hom^{proj}_{A}(X,X')$, there exist $\gamma:X\to P$ and $\delta:P\to X'$ morphisms of right $A$-modules such that $\delta\circ \gamma=\beta$. Since $P$ is finitely generated projective right $A$-module, given $y\in P$, $y=\sum y_ig_i(y)$ with $g_i:P\to A$ of right $A$-modules.\bigbreak 
	We will now consider $\alpha(x)=\sum (t\circ g_i\circ \gamma)(x)\delta(y_i) $ with $t$ the trace of $A$ and prove that $Tr_{A}(\alpha)=\beta$:\\
	\begin{align*}
		Tr_{A}(\alpha)(x)&=\sum_j \alpha\bigl(xe^j\bigr)e_j=\sum_j \sum_i (t\circ g_i\circ \gamma)(xe^j)\delta(y_i)e_j \\
		&=\sum_j \sum_i t (g_i( \gamma(x)e^j))\delta(y_i)e_j\\
		&=\sum_j \sum_i t ((g_i (\gamma(x))e^j)\delta(y_i)e_j\\
		&=\sum_i \delta(y_i)\sum_j t ((g_i (\gamma(x))e^j)e_j\\
		&=\sum_i \delta(y_i)g_i (\gamma(x))\\
		&=\delta\left(\sum_i y_i g_i\bigl(\gamma(x)\bigr)\right)\\
		&=\delta(\gamma(x))=\beta(x).
	\end{align*}
	Now we will start with $\beta\in Im(Tr_{A})$ and prove that  $\beta\in Hom^{proj}_{R}(X,X')$. \\Let $\beta(x)=Tr_{A}(\alpha)(x)=\sum_j \alpha(xe^j)e_j$ for some $\alpha\in Hom_{R}(X,X')$. Since $X$ is a finitely generated projective $R$-module then $\displaystyle{xe^j=\sum_{i=1}^{n}f_i\bigl(xe^j\bigr)x_i}$ with $f_i:X\to R$ of $R$-modules and 
	$$\beta(x)=\sum_j\alpha\bigl(xe^j\bigr)e_j=\sum_j \alpha\Bigl(\sum_if_i(xe^j)x_i\Bigr)e_j=\sum_{i,}\sum_{j}f_i(xe^j)\alpha(x_i)e_j.$$
	Let us now define $\gamma:X\to A^n$ such that $$\gamma(x)=\Bigl(\sum_{j}f_1(xe^j)e_j, \sum_{j}f_2(xe^j)e_j,\cdots , \sum_{j}f_n(xe^j)e_j\Bigr).$$
	To see that $\gamma$ is a right $A$-module morphism we use that $\Delta_A$ is an $A$-bimodule morphism, that is
	$$\begin{array}{rcl}
		\gamma(x\cdot a)&=&\Bigl(\sum_{j}f_1(xae^j)e_j, \sum_{j}f_2(xae^j)e_j,\cdots , \sum_{j}f_n(xae^j)e_j\Bigr)\\
		&=&\Bigl(\sum_{j}f_1(xe^j)e_ja, \sum_{j}f_2(xe^j)e_ja,\cdots , \sum_{j}f_n(xe^j)e_ja\Bigr)\\
		&=&\gamma(x)\cdot a
	\end{array}$$
	
	Let be $\delta:A^n\to X'$ defined as $$\delta(a_1, a_2,\cdots, a_n)=\sum_{i=1}^{n} \alpha(x_i)a_i.$$ 
	It is clear that $\delta$ is a right $A$-module morphism and $\beta=\delta\circ \gamma$.\\
	Finally, as $A^n$ is finitely generated projective right $A$-module, we conclude that $\beta\in Hom^{proj}_{A}(X,X')$.
\end{proof}
\begin{prop}
	Let $A$ be a Frobenius $R$-algebra with Casimir element $C_A=\sum_{i} e^i\otimes e_i$. Let $X$ be a finitely generated $R$-module. Then $X$ is a projective right $A$-module if and only if there exists a morphism of $R$-modules $\alpha$ such that $x=\sum \alpha\bigl(xe^i\bigr)e_i$,  $\forall x\in X$.
\end{prop}
\begin{proof}
	($\Rightarrow$) Consider $X$ a projective right $A$-module. Trivially $Id_{X}\in Hom^{proj}_{R}(X,X)$ and using the above lemma there exists $\alpha\in End_{R}(X)$ such that $Id_X=Tr_{A}(\alpha)$ and $\sum \alpha\bigl(xe^i\bigr)e_i=x$ for all $x\in X$.\\
	($\Leftarrow$) We have that $Tr_{A}(\alpha)=Id_X$ for some $\alpha\in End_{R}(X)$ so $Id_{X}\in Hom^{proj}_{R}(X,X)$. This means that there exist $P$ a finitely generated projective right $A$-module and right $A$-module morphisms $f:X\to P$, $g:P\to X$ such that $g\circ f=Id_X$. This information allows us to prove in a simple way that $X$ is a right projective $A$-module.
	
\end{proof}

 Next we will prove a weaker version of Proposition 4.4 of \cite{B} with $A$ and $B$ Frobenius algebras.  
 Note that if we take $A$ and $B$ non-symmetric Frobenius algebras in condition 2 we can only say that the morphism splits as right $A$-module morphism.
	
	\begin{prop}\label{prop14}
	Let $A$ and $B$ Frobenius algebras and $\phi:A\twoheadrightarrow  B$ an epimorphism of algebras.  The following conditions are equivalent:
	\begin{enumerate}
		\item $s_{\phi}$ is invertible in $B$.
		\item $\phi:A\twoheadrightarrow  B$ splits as a right $A$-module morphism.
		\item $B$ is a projective right $A$-module.
		\item Every projective $B$-module is a projective right $A$-module.
	\end{enumerate}
\end{prop}
\begin{proof}
	For (1) $\Rightarrow$ (2) consider $C_{A}=\sum e^i\otimes e_i$  and define the map $\sigma:B\to A$ such that $$\sigma(b)=\sum \varepsilon_B\bigl(bs_{\phi}^{-1}\phi(e^i)\bigr)e_i.$$
	By definition, $\phi\circ \sigma(b)=\sum \varepsilon_B\bigl(bs_{\phi}^{-1}\phi(e^i)\bigr)\phi(e_i)$.\bigbreak
	Since $(\phi\otimes \phi)(C_A)=s_{\phi}C_B$ and $s_{\phi}$ is invertible we deduce that $s_{\phi}^{-1}(\phi\otimes \phi)(C_A)=C_B$.
	This means that $C_B=\sum f^i\otimes f_i=\sum s_{\phi}^{-1}\phi(e^i)\otimes \phi(e_i)$. Using that $\sum \varepsilon_B(bf^i)f_i=b$ we conclude that
	$\sum \varepsilon_B(bs_{\phi}^{-1}\phi(e^i))\phi(e_i)=b$ and $\phi$ splits.\bigbreak
To prove that $\sigma$ is a right $A$-module morphism we use that $s_\phi\in ZB$.
	$$\sigma\bigl(b\cdot a\bigr)=\sigma\bigl(b\phi(a)\bigr)=\sum\varepsilon_B\bigl(b\phi(a)s_\phi^{-1}\phi\bigl(e^i\bigr)\bigr)e_i=
	\sum\varepsilon_B\bigl(bs_\phi^{-1}\phi(a)\phi\bigl(e^i\bigr)\bigr)e_i$$
	$$=\sum\varepsilon_B\bigl(bs_\phi^{-1}\phi\bigl(ae^i\bigr)\bigr)e_i=\sum\varepsilon_B\bigl(bs_\phi^{-1}\phi\bigl(e^i\bigr)\bigr)e_ia=\sigma(b)\cdot a.$$
	\\
	For (2) $\Rightarrow$ (3) we need to prove that $B$ is a projective right $A$-module considering the action $b\cdot a=b\phi(a)$. Given an epimorphism of A-modules $f:M\to N$ and an epimorphism $g:B\to M$ we need to find $h:B\to M$, a morphism of $A$-modules such that $f\circ h=g$. Consider $g \circ\phi:A\to N$. Since $A$ is a projective $A$-module then there exists $p:A\to M$ such that $f\circ p=g \circ\phi$. If $h=p\circ\sigma$ then $f\circ h=f \circ p\circ\sigma=g \circ\phi\circ \sigma=g$.	\bigbreak
	(3) $\Rightarrow$ (4) is a known result.\bigbreak
	Finally, we will prove (4) $\Rightarrow$ (1). \\ Since $B$ is a projective $A$-module then there exists $\alpha\in End_R(B)$ such that $Tr_{A}(\alpha)=Id_{B}$, that is $b=\sum \alpha(b\cdot e^i)\cdot e_i$ for all $b\in B$.\\ On the other hand, since $$\sum \phi(e^i)\otimes \phi(e_i)=s_{\phi}\sum f^j\otimes f_j=\sum f^j\otimes f_js_{\phi}$$ we have that 
	\begin{align*}
	Tr_A(\alpha)(b)&=\sum \alpha(b\cdot e^i)\cdot e_i\\
	&=\sum \alpha(b\phi(e^i))\phi(e_i)\\
	&=\sum \alpha(bf^j)f_js_{\phi}\\
	&=Tr_{B}(\alpha)(b)s_{\phi},
	\end{align*} 
	
	so $Tr_{B}(\alpha)(1_B)s_{\phi}=Tr_A(\alpha)(1_B)=Id_B(1_B)=1_B$ and  $s_{\phi}$ is invertible.
	
	\end{proof}
It is natural to ask if all the equivalences in Proposition \ref{prop14} hold in the case $A$ symmetric nearly Frobenius or just nearly Frobenius.  The following results and examples answer this questions.\bigbreak
The following example shows that Proposition \ref{prop14} is not valid if we require $A$ to be a non-symmetric nearly Frobenius algebra.
\begin{example}
	Consider $\displaystyle{A=\frac{\mathbb{K}Q_1}{I_1}}$ and $\displaystyle{B=\frac{\mathbb{K}Q_2}{I_2}}$ with $Q_1$ and $Q_2$:\\
	$$Q_1= \xymatrix{
			\bullet_1 \ar@(ul,ur)[]^{\alpha} \ar[r]_\beta &
			\bullet_2 
		}\quad\mbox{and}\quad Q_2= \xymatrix{
			\bullet_1 \ar@(ul,ur)[]^{\alpha} 
		}$$	
	$I_1=\langle\alpha^2, \beta\alpha\rangle$ and $I_2=\langle\alpha^2\rangle$.\\
	The algebra $B$ is a symmetric algebra with Casimir element $C_{B}=1\otimes \alpha+\alpha\otimes 1$, and counit $\varepsilon:B\to \mathbb{K}$ defined as $\varepsilon(e_1)=1$ and $\varepsilon(\alpha)=0$.\\
	The algebra $A$ is a nearly Frobenius algebra and the coproducts have the form $\de_{A}(1)=a\alpha\otimes \alpha+a'\alpha \otimes \beta$ where $a,a'\in \mathbb{K}$.\\	
	Consider the epimorphism of algebras $\phi:A\to B$ with $\phi(e_1)=e_1$, $\phi(e_2)=0$, $\phi(\alpha)=\alpha$ and $\phi(\beta)=0$.\\
	Then $s_{\phi}=\varepsilon(a\alpha)\alpha+\varepsilon(a'\alpha)0=0$ so is not invertible but if we consider the inclusion of 
	$Q_2$ in $Q_1$ $i:B\to A$ the morphism $\phi$ splits. 	
\end{example}	

\begin{prop}\label{prop2}
	Let $A$ be a symmetric nearly Frobenius algebra and $B$ symmetric algebra with $\phi: A\twoheadrightarrow  B$ an epimorphism of algebras. Then if $s_{\phi}$ is invertible$,\phi$ splits as a morphism of $A$-bimodules.
\end{prop}
\begin{proof}
We can use the same map as the Proposition \ref{prop14}, that is $\sigma:B\to A$ such that $$\sigma(b)=\sum \epsilon_B(bs_{\phi}^{-1}\phi(e^i))e_i.$$
We only need to prove that $\sigma$ is a morphism of left $A$-modules.
	$$\sigma(a\cdot b)=\sigma\bigl(\phi(a)b\bigr)=\sum \varepsilon_B\bigl(\phi(a)bs_{\phi}^{-1}\phi(e^i)\bigr)e_i=\sum \varepsilon_B\bigl(bs_{\phi}^{-1}\phi(e^i)\phi(a)\bigr)e_i.$$
The last identification is a consequence of the symmetry property of $B$. On the other hand, using that $\phi$ is a morphism of algebras we have that:
$$\sigma(a\cdot b)=\sum \varepsilon_B\bigl(bs_{\phi}^{-1}\phi(e^ia)\bigr)e_i.$$
The symmetry property on $A$ means that $\sum e^i\otimes e_i=\sum e_i\otimes e^i$, using that $\Delta$ is a morphism of $A$-bimodules we have 
$$\sum ae^i\otimes e_i=\sum e^i\otimes e_ia=\sum ae_i\otimes e^i=\sum e_i\otimes e^ia. $$
Then $$\sigma(a\cdot b)=\sum \varepsilon_B\bigl(bs_{\phi}^{-1}\phi(e^ia)\bigr)e_i=\sum \varepsilon_B\bigl(bs_{\phi}^{-1}\phi(e^i)\bigr)ae_i=a\sum\varepsilon_B\bigl(bs_{\phi}^{-1}\phi(e^i)\bigr)e_i=a\sigma(b).$$
\end{proof}

The following example shows that a version of Proposition \ref{prop14} in the context of Proposition \ref{prop2}  does not hold.

\begin{example}
	 Let be $\displaystyle{A=\frac{\k[x]}{x^2}}$ and $\Delta_1(1)=x\otimes x$, $\Delta_2(x)=x\otimes 1+ 1\otimes x$. Note that $\bigl(A,\Delta_1\bigr)$ is a symmetric nearly Frobenius algebra and $\bigl(A, \Delta_2\bigr)$ is a symmetric algebra. In this case we consider $\phi=\operatorname{Id_A}:A\to A$. For this map the Schur element is
	 $$s_{\operatorname{Id_A}}=\sum\varepsilon_A\bigl(e^i\bigr)e_i=\varepsilon_A(x)x=x.$$
	 The Schur element is not invertible in $A$, but the map $\operatorname{Id_A}$ split as a morphism of $A$-bimodule.
\end{example}

\bibliographystyle{elsarticle-harv}

\end{document}